\documentclass[fleqn]{article}

\usepackage[utf8]{inputenc}
\usepackage[british]{babel}
\usepackage{amsmath}
\usepackage{amssymb}
\usepackage{amsfonts}
\usepackage{amsthm}
\usepackage{mathrsfs}
\usepackage{fancyhdr}
\usepackage{graphicx}
\usepackage{caption}
\usepackage{tabulary}
\usepackage{multirow}
\usepackage{makecell}
\usepackage[export]{adjustbox}
\usepackage{subfig}
\usepackage{wrapfig}
\usepackage{hyperref}
\usepackage{float}
\usepackage{fullpage}
\usepackage{enumitem}
\usepackage{stmaryrd}
\usepackage[T1]{fontenc}
\usepackage{tikz}
\usetikzlibrary{arrows}
\usetikzlibrary{cd}

\hypersetup{
backref=true,
pagebackref=true,
hyperindex=true,
colorlinks=true,
breaklinks=true,
urlcolor= blue,
linkcolor= blue,
bookmarks=true,
bookmarksopen=false
}

\title{Extending a Morse function to a non-orientable $3-$manifold}
\date{}
\author{Clément Laroche\\supervised by Thomas Fiedler}

\providecommand{\keywords}[1]{\textbf{\textit{Keywords---}} #1}

\newcommand\Kl{Kl_s^3}

\newcommand\paragraphtitle[1]{\begin{center}\textbf {#1}\end{center}}

\newcounter{refcounter}

\newenvironment{proofthm}[1][\unskip]{\noindent\hskip\labelsep {\itshape \bfseries Proof of Theorem #1:}}{\vspace{-\baselineskip}\begin{flushright}{$\square$}\end{flushright}}

\if 0
\newenvironment{proof}{\noindent\hskip\labelsep {\itshape \bfseries Proof:}}{\begin{flushright}{$\square$}\end{flushright}}
\newenvironment{definition}[1][\unskip]{\begin{trivlist}
\item[\hskip \labelsep {\bfseries Definition #1:}]}{\end{trivlist}}
\newenvironment{property}[1][\unskip]{\begin{trivlist}
\item[\hskip \labelsep {\bfseries Property #1:}]}{\end{trivlist}}
\newenvironment{proposition}[1][\unskip]{\begin{trivlist}
\item[\hskip \labelsep {\bfseries Proposition #1:}]}{\end{trivlist}}
\newenvironment{lemma}[1][\unskip]{\begin{trivlist}
\item[\hskip \labelsep {\bfseries Lemma #1:}]}{\end{trivlist}}
\newenvironment{theorem}[1][\unskip]{\begin{trivlist}
\item[\hskip \labelsep {\bfseries Theorem #1:}]}{\end{trivlist}}
\newenvironment{examples}{\begin{trivlist}
\item[\hskip \labelsep {\bfseries Examples:}]}{\end{trivlist}}
\newenvironment{remark}[1][\unskip]{\begin{trivlist}
\item[\hskip \labelsep {\bfseries Remark #1:}]}{\end{trivlist}}
\fi

\newtheorem{definition}{Definition}

\newtheorem{proposition}{Proposition}
\newtheorem{lemma}{Lemma}
\newtheorem{theorem}{Theorem}

\def\changemargin#1#2{\list{}{\rightmargin#2\leftmargin#1}\item[]}
\let\endchangemargin=\endlist 

\DeclareMathOperator{\Ker}{Ker}
\DeclareMathOperator{\Image}{Im}

\begin{document}
\maketitle{}
\begin{abstract}
Considering a solid 3-dimensional Klein bottle and a collaring of its boundary, can we extend a generic $C^\infty$ non-singular  function defined on the collaring to the full solid Klein bottle without critical points? We give a condition on the Reeb graph of the given function that is necessary and sufficient for the existence of such a non-singular extension.
\end{abstract}
\large

\paragraphtitle{Introduction}

By a \emph{(generic) Morse function} on a closed manifold, we mean a real-valued function with non-degenerate critical points and with distinct critical values. By a \emph{(generic) non-singular Morse function on a manifold with boundary}, we mean a real-valued function without any critical points and for which the restriction to the boundary is a generic Morse function.

Let $\Kl$ be the solid Klein bottle and let the usual Klein bottle $Kl^2$ be its boundary.
The problem is thus stated as  following: given the germ of a Morse function $f$ along $Kl^2$, is there a non-singular Morse function $F$ extending $f$ to $\Kl$ ?

We will also discuss the following weaker version of this problem: given the germ of a Morse function $f$ along a closed non-orientable manifold $M^2$, is there any compact manifold $N^3$ whose boundary is $M^2$ and a non-singular Morse function $F$ extending $f$ to $N^3$?

\vspace{0.5cm}
It is well known that extension problems with prescribed topology are usually very difficult. However, C. Curley\cite{Curley} has solved the above  problem in the case of the 2-sphere bounding the 3-ball. The same problem one dimension lower was solved by Blank and Laudenbach\cite{BlankLaudenbach}. The difficulty is the following: given the germ $f$ on the collaring we could always extend $f$ to a Morse function  on the 3-manifold. However, we should now try to cancel the critical points of the Morse function without changing it near the boundary. This runs into the usual problems in low-dimensional topology (mainly the failure of the Whitney trick) and has no chance at all to work . Therefore we construct  extensions on elementary slices of 3-manifolds which we glue then together instead of the whole extension. We can control the homotopy type of the 3-manifold under the gluing, but in order to fix the 3-manifold up to homeomorphism we have to use Perelman's theorem. (Curley's result was long before Perelman and he had to use instead Schönfliess theorem , which is still just known up to the 2-sphere in the 3-sphere). It is well known that there are e.g. 3-dimensional lens spaces which are homotopy equivalent but not homeomorphic. Let us take out a small 3-ball from such a lens space. It is clear  that our method  would not allow  to decide whether a germ $f$ on a collaring of the 2-sphere can be extended non-singularly to the lens space. So, it remains a challenging problem to extend our result to other 3-manifolds.

The non-singular extension problem of germs for spheres in arbitrary dimensions to the ball was studied by S. Barannikov\cite{Barannikov}. He has given a necessary condition for the existence of an extension. Valentin Seigneur has proven that Barannikov's condition is not sufficient for the circle and for the 2-sphere (unpublished Master thesis 2013 at the Université Paul Sabatier, Toulouse). 

M. Yamamoto\cite{Yamamoto} has given a necessary and sufficient condition for a germ $f$ on a collaring of the circle to extend to a non-singular function $F$ on a given surface $V$, but such that $F$ is induced by a height function on an immersion of $V$ into $\mathbb{R}^2$. Le Van Tu has proven that there are germs which can be extended in the sense of Blank-Laudenbach but which can not be extended in the stronger sense of Yamamoto (unpublished Master thesis 2016 at the Université Paul Sabatier, Toulouse). 

Finally, let us mention that our extension result is the first one in the case of non-orientable manifolds and involves interesting new features in the construction.

\vspace{0.5cm}
\begin{definition}
The \emph{Reeb graph} of a Morse function $f$, noted $dgm(f)$ (standing for « diagram of $f$ »), is the quotient space $M^2 / \sim$ where $x\sim y$ iff\\
$\left\{\begin{array}{ll}
f(x)=f(y)\\
x \text{ and } y \text{ belong to the same connected component of }f^{-1}(f(x))
\end{array}\right.$

\vspace{0.25cm}
The vertices of the Reeb graph are the classes of the critical points.

Furthermore, if $f$ is the germ of a Morse function along a manifold $M^2$, its \emph{labelled graph}, noted $dgm^+(f)$, is its Reeb graph with each vertex labelled "$+$" if $\langle \mathrm{d}f(p), \nu\rangle > 0$ or "$-$" if $\langle \mathrm{d}f(p), \nu\rangle < 0$ where $p$ is the corresponding critical point and $\nu$ is an outward-pointing normal at $p$.
\end{definition}

For the rest of the article, $f$ is the germ of a Morse function along a manifold $M^2$.

The vertices of $dgm(f)$ can be of the following form:

\begin{figure}[H]
\centering
\caption{\label{DgmVert} Vertex Configurations of dgm(f)}
\includegraphics[width=380px]{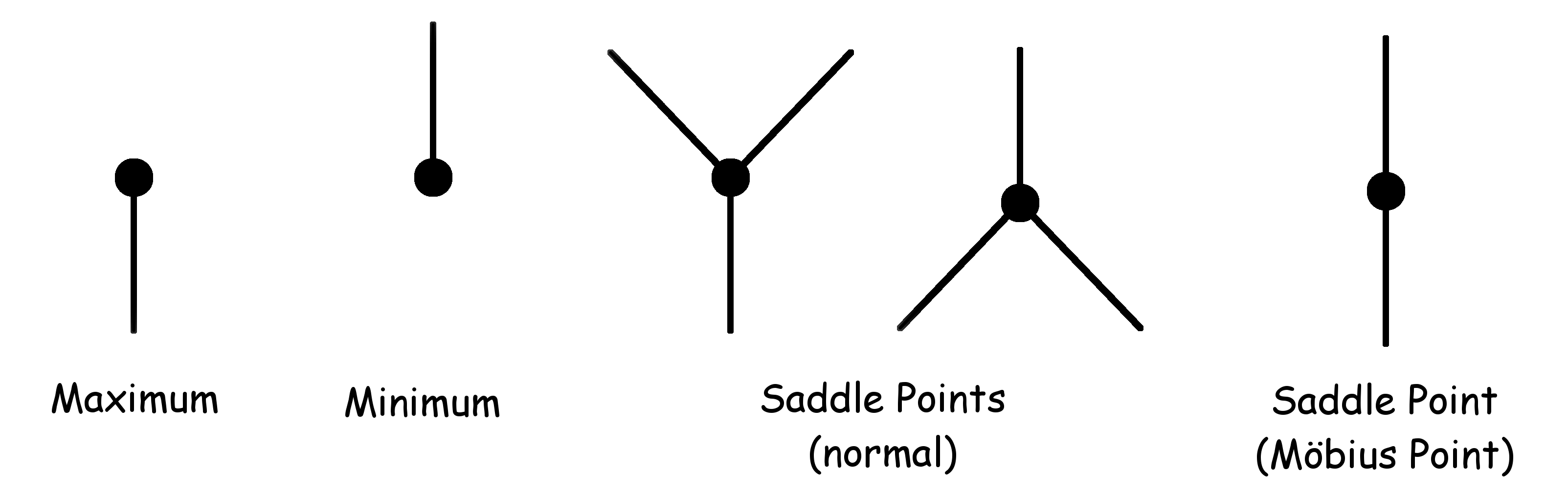}
\end{figure}

The Möbius points are the ones arising from the non-orientability of the manifold, though there exists germs of Morse functions along non-orientable manifolds with no such points. Their effect on the level surface of an extension is to create or remove a Möbius strip. That means that their presence on the Reeb graph of $f$ implies the existence of non-orientable level surfaces for any extension $F$. We will see that the converse is false in general.

\begin{wrapfigure}[18]{r}[38pt]{160px}
\centering
\caption{\label{CollapseSample} Z Coordinate as an Extendable Function\\Note that the level surfaces of the extension are all orientable discs}
\includegraphics[width=160px]{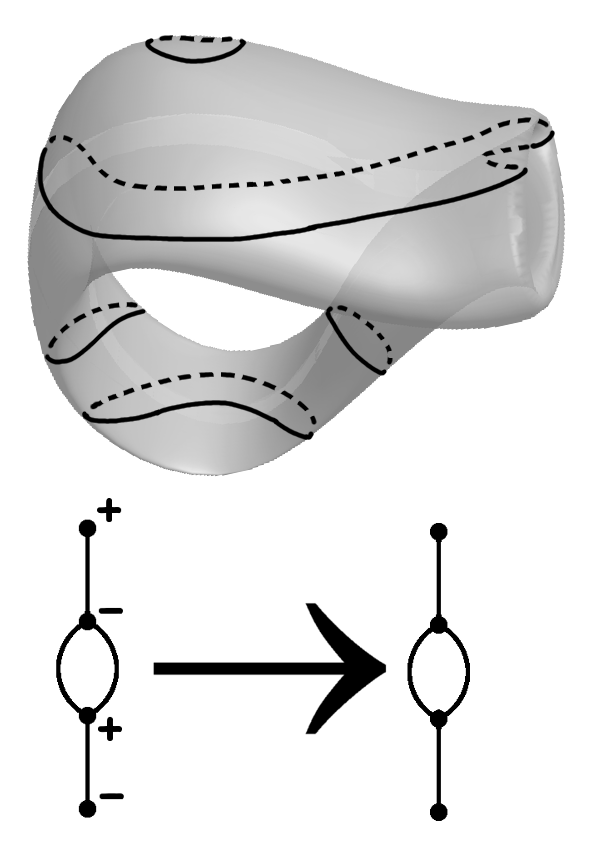}
\end{wrapfigure}

\begin{definition}
A \emph{collapse} of $dgm^+(f)$ is a surjective map $C:dgm^+(f) \rightarrow dgm(h)$ onto another diagram commuting with the height function and bijectively mapping the vertices of $dgm^+(f)$ onto those of $dgm(h)$.
\end{definition}

C. Curley proved\cite{Curley} that when the manifold is orientable, the existence of a non-singular extension is equivalent to the existence of a "nice" collapse of the labelled graph, that is, a particular mapping to another graph (the point is that this other graph is the Reeb graph of the sought extension). Our contribution here is to generalize his results to the non-orientable case.

If $F$ is an extension, we note $dgm(F)$ its Reeb graph. The "nice" collapse needs to verify different kind of properties to guarantee that it is indeed a collapse onto the Reeb graph of an extension:
\begin{itemize}
\item restrictions around the vertices, indicating the local behavior of the Morse function around the critical points (depicted in the table below),
\item a "loop resolution" property, guaranteeing that the non-trivial loops of level surfaces that may be created by some critical points are properly cleaned by other critical points,
\item an "orientability consistency" property, guaranteeing that the orientability of level surfaces does not need to change between two critical values,
\item in the case of the Klein bottle problem, an evidence that the fundamental group of the domain of definition of $F$ is $\mathbb{Z}$ and thus that it is indeed a solid Klein bottle.
\end{itemize}

\paragraphtitle{The Nessecary and Sufficient Conditions}

\begin{wrapfigure}[12]{r}[34pt]{180px}
\centering
\caption*{\label{Tab1} Local Collapses}
\includegraphics[width=200px]{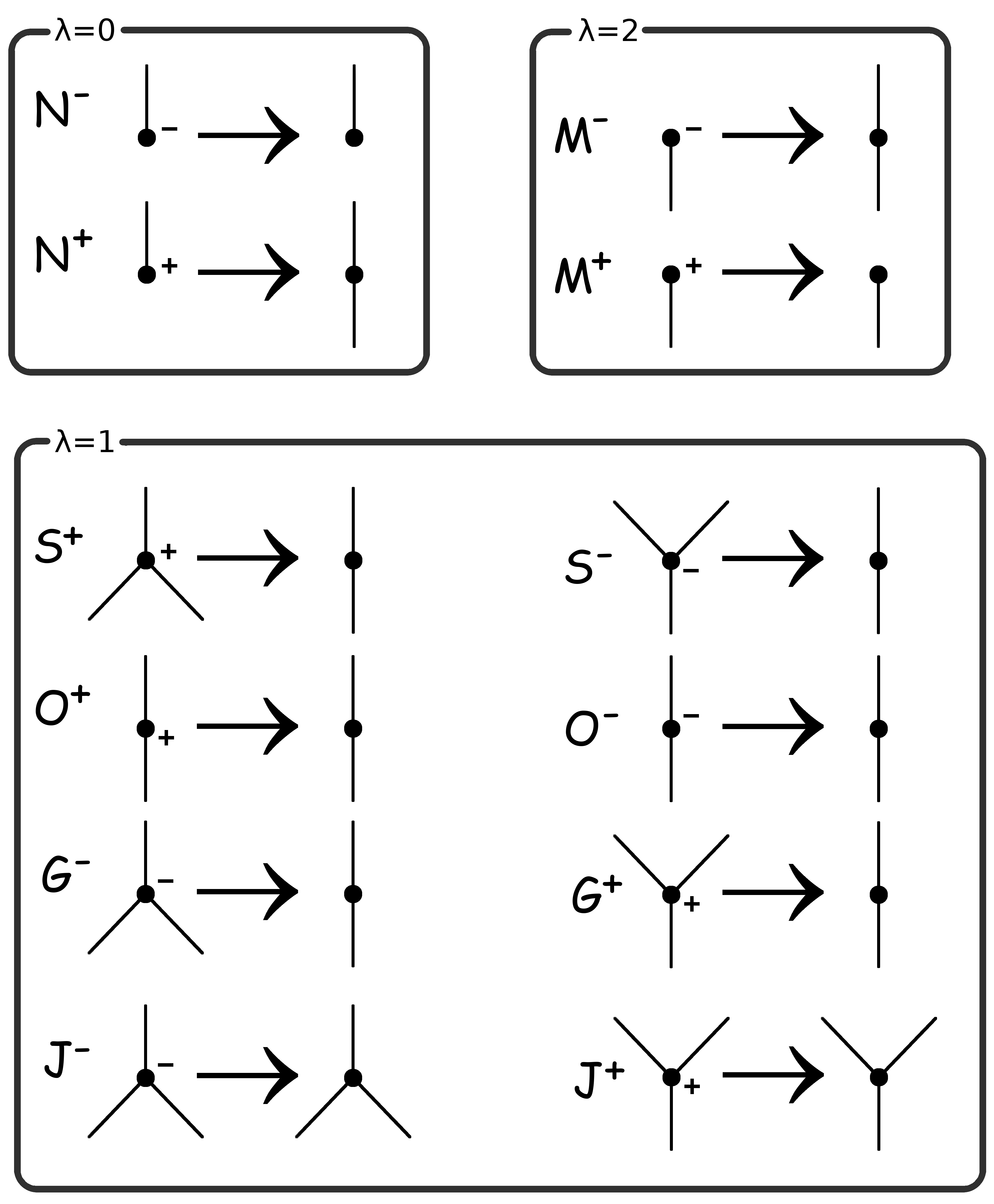}
\end{wrapfigure}

\hspace{0pt}
\begin{definition}
A collapse $C$ is said to be \emph{locally allowable} if it satisfies the restrictions on the vertices listed in the following table ($\lambda$ is the Morse index of the critical point for $f_{|M^2}$). Furthermore, in the cases $J^+$ and $J^-$, $C$ must be a local homeomorphism in a neighborhood of the vertices.
\end{definition}

%\begin{table}[H]
%\centering
%\caption{\label{Tab1} Local Collapses}
%\includegraphics[width=200px]{DgmFullNoLabel.png}
%\end{table}

\vspace{0.25cm}
We recall a version of the compact connected surface classification Theorem, which will be of great use to keep track of modifications:

\begin{theorem}\cite{Hirsch}\refstepcounter{refcounter}\label{ThmSurfaceClass}
A compact connected surface $S$ with $b$ boundary components ($b>0$) is diffeomorphic to either one of the following surfaces:
\begin{enumerate}[label={(\arabic*)}]
\item a closed disc with $b-1$ open discs removed from it and $h$ handles attached to it ($h \ge 0$) such that it is orientable,
\item the connected sum of an orientable surface of type (1) and a projective plane $\mathbb{P}^2(\mathbb{R})$,
\item the connected sum of an orientable surface of type (1) and a Klein bottle $Kl^2$.
\end{enumerate}
\end{theorem}

\begin{wrapfigure}[9]{l}[34pt]{210px}
\captionsetup{justification=centering}
\centering
\caption{\label{FigSurfaceClass} Surface classes}
\includegraphics[width=200px]{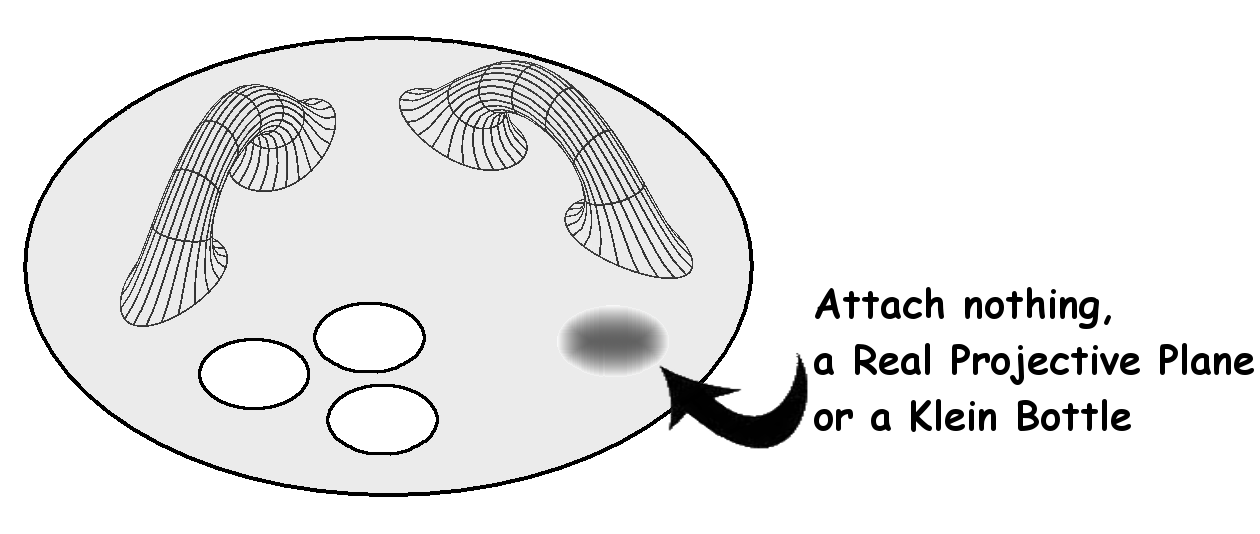}
\end{wrapfigure}

In the case (1), $h$ is the usual genus of the (orientable) surface $S$.\\
In the cases (2) and (3), $S$ is non-orientable and the rank of the fundamental group, minus the part due to the removed open discs, are $2h+1$ and $2h+2$ respectively.

This means that level surfaces are caracterized by their number of boundary components, their orientability and one more integral number. We could take the Euler caracteristic for that, but it will be a bit more efficient to use the following.

\begin{definition}
The \emph{demigenus} of a compact surface $S$ with boundary is the number\\
$g := 2-\chi(\hat{S})$ where $\chi$ is the Euler characteristic and $\hat{S}$ is the surface $S$ whose holes are filled up with discs.
\end{definition}

%%$g := 2-\chi(S)-\pi_0(\partial S)$ where $\chi$ is the Euler characteristic and $\pi_0(\partial S)$ is the number of connected components of the boundary of $S$.

It is the rank of the fundamental group of the surface obtained by attaching discs to the boundary components of $S$. In the orientable case, $g$ equals the double of the usual genus. Note that if the demigenus is odd, then the surface is non-orientable. Also, if the demigenus is $0$, then the surface is orientable. We sum up that by saying the couple $(g, o)$ of the demigenus and the orientability of a surface belongs to the following set:\\
${\Lambda := (2\mathbb{N^*}\times\mathbb{Z}/\raisebox{-0.3ex}{$2\mathbb{Z}$})\cup((2\mathbb{N}+1)\times\{1\})\cup (0,0)}$\\
where $o=0$ stands for an orientable surface and $o=1$ stands for a non-orientable surface.

In order to have a proper collapse from the Reeb graph of $f$ to the Reeb graph of an extension, we need to keep track of the demigenus and the orientability of the level surfaces. We do so by labelling the edges of the collapse $dgm(h)$ with an element of $\Lambda$, representing the expected demigenus and orientability of the extension.

\begin{definition}
A \emph{labelled collapse} is a collapse to a diagram denoted by $dgm^{g,o}(h)$ whose edges are labelled by elements of $\Lambda$.

A \emph{(globally) allowable collapse} is a labelled collapse locally allowable and satisfying the local restrictions of the table \ref{TabFullLabelledCollapse} below.
\end{definition}

\begin{table}[H]
\centering
\caption{\label{TabFullLabelledCollapse} Labelled Local Collapses}
\includegraphics[width=380px]{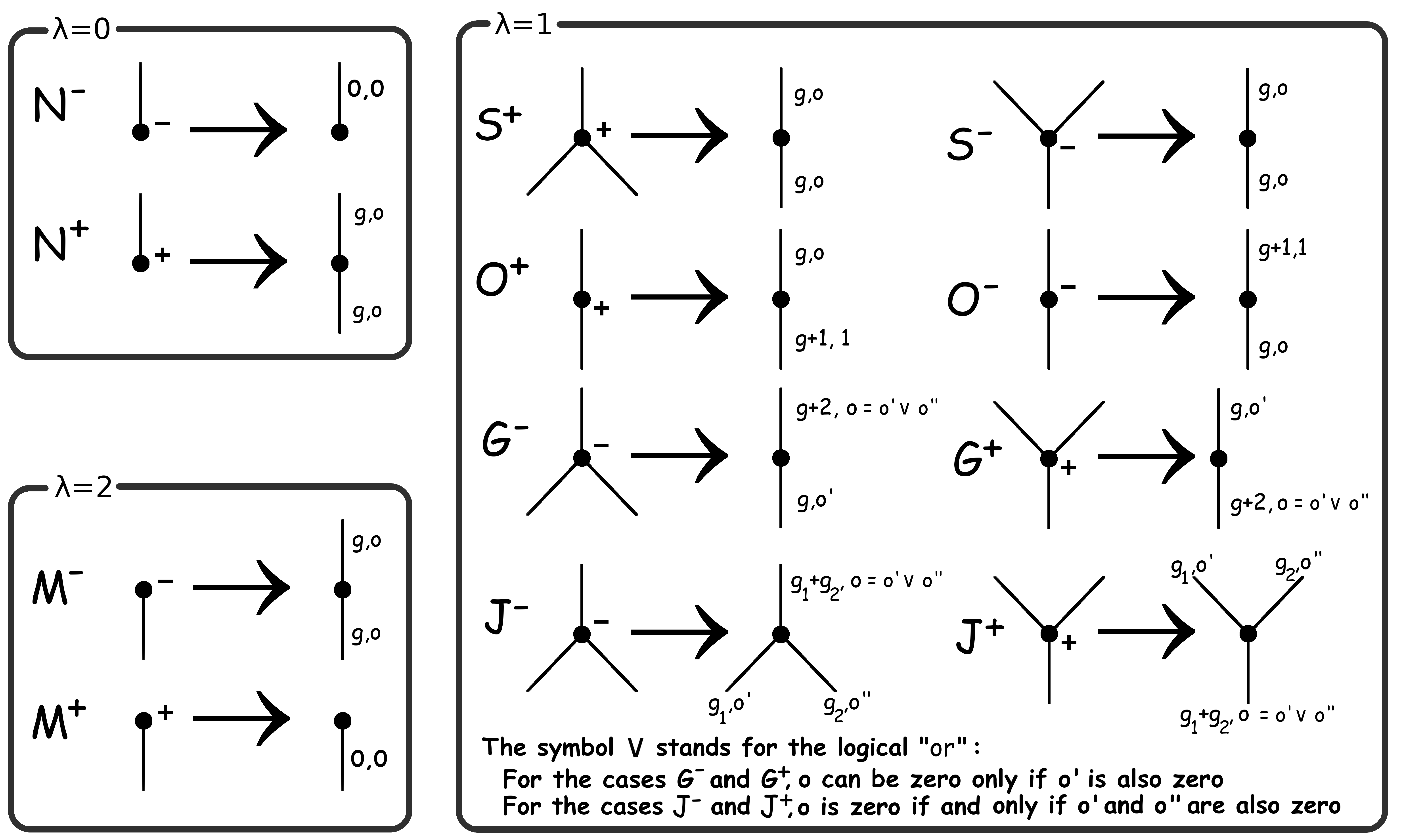}
\end{table}

Now we can state the announced results.

\begin{theorem}\refstepcounter{refcounter}\label{ThmAbstractVersion}
Let $f$ be the germ of a Morse function along a closed manifold $M^2$.
Then there exists a compact manifold $N^3$ bounded by $M^2$ and a non-singular Morse function $F$ extending $f$ to $N^3$ iff
there exists an allowable collapse from $dgm^+(f)$ to a labelled graph $dgm^{g,o}(h)$.
\end{theorem}

\begin{theorem}\refstepcounter{refcounter}\label{ThmKleinVersion}
Let $f$ be the germ of a Morse function along $Kl^2$. Then $f$ can be extended to a non-singular Morse function $F$ to $\Kl$ iff
$dgm^+(f)$ has an allowable collapse $dgm^{g,o}(h)$ verifying either one of those conditions:
\begin{enumerate}[label={(\arabic*)}]
\item The label $o$ is 0 for all the edges and $\pi_1(dgm^{g,o}(h)) \simeq \mathbb{Z}$,
\item The set of points belonging to an edge labelled with $o=1$ is connected and non-empty, and $dgm^{g,o}(h)$ is simply connected.
\end{enumerate}
Moreover, one can choose the extension $F$ such that its Reeb graph is $dgm(h)$ and the demigenus and orientability of its level surfaces match the labels $g$ and $o$ respectively.
\end{theorem}

\newpage

\emph{Examples: }
\begin{figure}[H]
\captionsetup{justification=centering}
\centering
\caption{\label{Example1} One of the simplest extendable germ of a Morse function from $Kl^2$ to $\Kl$}
\includegraphics[width=100px]{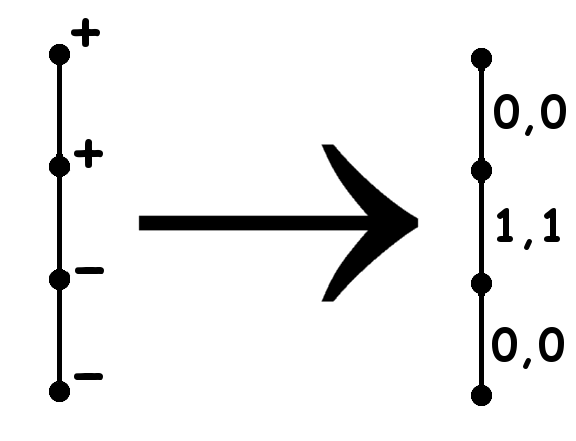}
\caption{\label{Example2} Another extendable germ of a Morse function from $Kl^2$ to $\Kl$\\Change the label of any one of the vertices and the germ is not extendable anymore}
\includegraphics[width=150px]{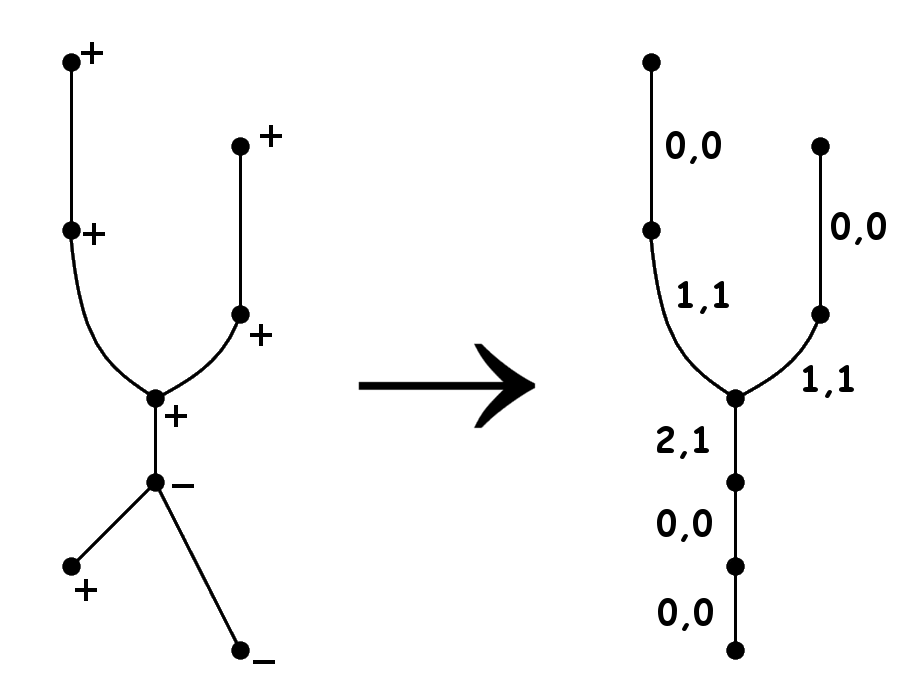}
\caption{\label{Example3} An extendable germ of a Morse function from $Kl^2$ to $\Kl$ with the level surfaces of an extension\\The thick lines are the level curves along the boundary}
\includegraphics[width=500px]{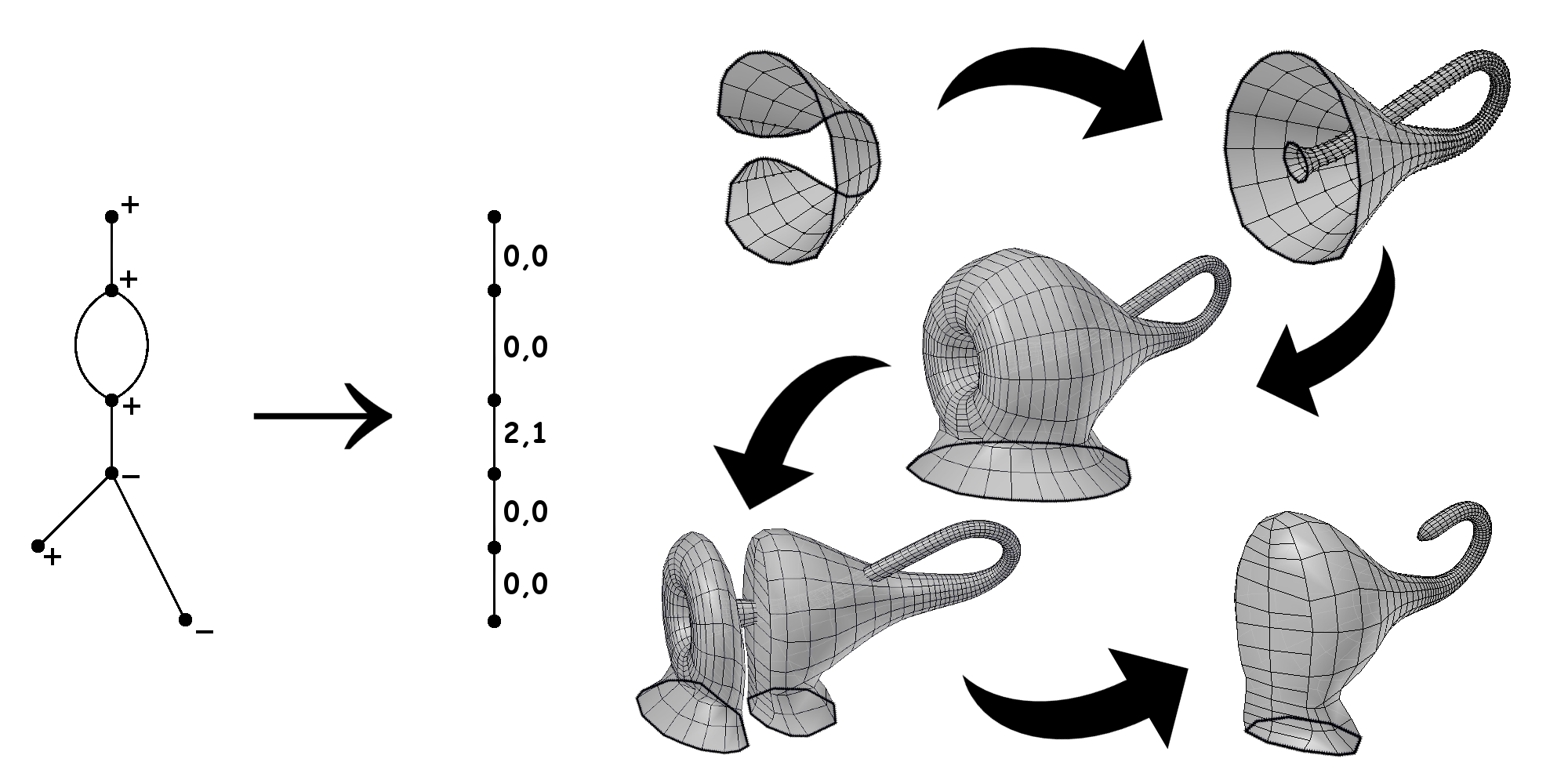}
\end{figure}

\newpage

\paragraphtitle{The Proofs}

\begin{proofthm}[\ref{ThmAbstractVersion}]

$\implies$ : Let $F$ be a non-singular extension of $f$.
It is easy to see that, thanks to the distinct critical values hypothesis, $F$ induces a locally allowable collapse on the Reeb graphs.

Notice that the connected components of the level surfaces close to points of type $M^+$ or $N^-$ (they are sometimes called \emph{true extremums}) are discs $D^2$. Passing through a critical value modifies the topology of this surface and all the modifications done on the level surfaces lead back to a disc. The fact that $F$ is non-singular guarantees that those level surface modifications only happen at critical value of $f_{|M^2}$. The following table shows the effects of the local restrictions of the collapse to the different useful indicators. We ignore the restrictions whose effect is just to split the level surfaces.

\begin{table}[H]
\centering
\caption{\label{LocalEffects} Effects of local restrictions on the topological invariants}
\changemargin{-1cm}{1cm}
\vspace{-0.5cm}
\begin{tabular}{|c|c|c|c|}
\hline
\multirow{2}{*}{Critical point} & number of connected components & Demigenus & Orientability\\
 & of the boundary & of the level surface & of the level surface\\
\hline
$S^+$ and $M^-$ & $+1$ & $0$ & \multirow{2}{*}{No change} \\
$S^-$ and $N^+$ & $-1$ & $0$ & \\
\hline
$O^+$ & $0$ & $+1$ & Leads to a non-orientable surface\\
$O^-$ & $0$ & $-1$ & Can only happen on a non-orientable surface\\
\hline
\multirow{2}{*}{$G^+$} & \multirow{2}{*}{$-1$} & \multirow{2}{*}{$+2$} & Either leads from orientable to non-orientable\\
& & & or does not change the orientability\\
\multirow{2}{*}{$G^-$} & \multirow{2}{*}{$+1$} & \multirow{2}{*}{$-2$} & Either leads from non-orientable to orientable\\
& & & or does not change the orientability\\
\hline
\end{tabular}
\endchangemargin
\changemargin{1cm}{-1cm}
\caption*{The changes are from the connected component of the level surface above the critical point to the one below the critical point.}
\endchangemargin
\end{table}

\vspace{-0.5cm}
Now, considering $dgm(F)$ labelled with both the demigenus and the orientability flag of its level surface components, the induced collapse is allowable. Indeed, since the level surface components are discs around the extrema, the cases $M^+$ and $N^-$ are labelled with $(0, 0)$ as required. The critical points of types $M^-$, $N^+$, $S$, $G$ and $O$ satisfy the related local change of labelling because of their effect depicted in table \ref{LocalEffects}. And the critical points of types $J$ do not change the demigenus of the whole level surfaces. In case $J^-$, a Möbius ribbon in the level surface component of one of the two components below the critical point can be moved upward to the one above. Reciprocally, one cannot realize a non-orientable surface as the boundary-connected-sum of two orientable ones. So the orientability satisfies the restrictions of the table \ref{TabFullLabelledCollapse} as well.

\vspace{0.75cm}
$\impliedby$ : As in Curley's article, we construct the extension and its domain step by step. Let $(v_1, ..., v_n)$ and $(w_1, ..., w_n)$ the vertices of $dgm^+(f)$ and $dgm^{g,o}(h)$ respectively, in the increasing order. Starting with $M_n = MOD(M^+)$ and attaching one of the local models below at each step, we construct a sequence of manifold $M_1, ..., M_n$ such that $M_i \subset M_{i-1}$ and maps $F_i$ that are non-singular functions on $M_i$ such that $F_{i\, | M_{i+1}} = F_{i+1}$ and $F_1$ is the extension of $f$.

\begin{figure}[H]
\centering
\caption{\label{Models} Local Models}
\includegraphics[width=260px]{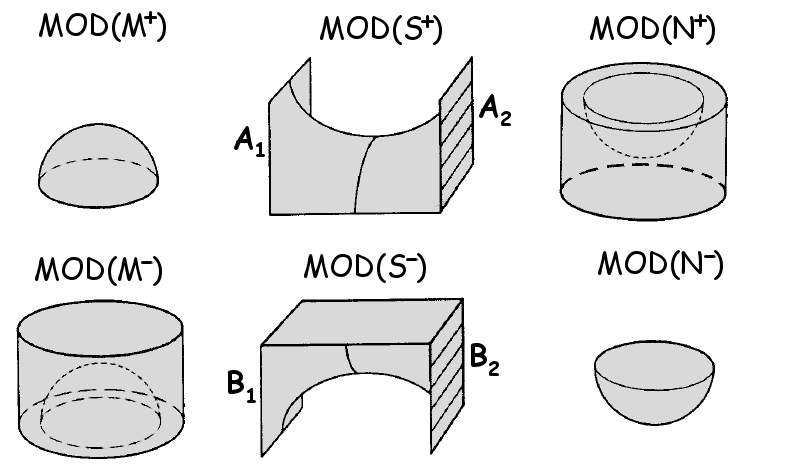}
\end{figure}

Given $M_{i+1}$ and $F_{i+1}$, we construct $M_i$ and $F_i$ by cases. We can speak of level surface components that are above $w_i$ (but not below) since that part of the extension is already constructed. We will note $L_{i+1}$ the whole level surface minimizing $F_{i+1}$. When we mention surfaces above $w_i$ (resp. curves above $v_i$), we implicitly consider connected components of $L_{i+1}$ (resp. $\partial L_{i+1}$). The way we define $F_i$ from $F_{i+1}$ will be clear from the construction of $M_i$.

If $v_i$ is of type $M^+$, let $M_i = M_{i+1} \cup_{g_1} (L_{i+1} \times I) \coprod MOD(M^+)$ where $I = [-1, 1]$ and $g_1$ maps $L_{i+1}$ to $L_{i+1} \times \{1\}$.

If $v_i$ is of type $M^-$, let $L_a$ be above $w_i$ and $D^2$ a disc in $L_a$.\\
We set ${M_i = M_{i+1} \cup_{g_1} (\overline{L_{i+1}\backslash D^2}) \times I \cup_{g_2} MOD(M^-)}$ where $g_2$ maps $\partial D^2 \times I$ to the cylindrical side of $MOD(M^-)$.

\begin{wrapfigure}{r}{140px}
\captionsetup{justification=centering}
\centering
\caption{\label{ProofSample1} The Case $S^+$}
\includegraphics[width=120px]{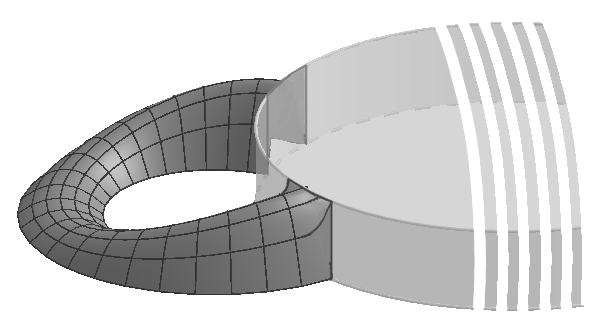}
\end{wrapfigure}

If $v_i$ is of type $S^+$, let $\partial_a$ be above $v_i$ (that is, a component of $\partial L_{i+1}$ above $v_i$) and $\alpha_1$, $\alpha_2$ be disjoint arcs in $\partial_a$.
We set ${M_i = M_{i+1} \cup_{g_1} (L_{i+1} \times I) \cup_{g_3} MOD(S^+)}$ where $g_3$ maps $(\alpha_1 \cup \alpha_2) \times I$ to the sides $A_1 \cup A_2$ of $MOD(S^+)$. The map $g_3$ must also verify an orientation compatibility condition: for all rectangles $R \subset L_{i+1}$ joining $\alpha_1$ and $\alpha_2$ and such that each side of $R$ is homotopic to some arc of $\partial_a$, we have that $(R \times I) \cup_{g_3} MOD(S^+)$ is orientable.

If $v_i$ is of type $O^+$, we do the same as for the case $S^+$ but with the orientation non-compatibility condition: for all rectangles $R \subset L_{i+1}$ like above, we have that $(R \times I) \cup_{g_3} MOD(S^+)$ is not orientable.

If $v_i$ is of type $J^+$, we do the same as for the case $S^+$ but with $\alpha_1$ and $\alpha_2$ being respectively arcs in $\partial_{a_1}$ and $\partial_{a_2}$, the two connected components above $v_i$. Also, the orientability condition becomes the following:
\begin{itemize}
    \item If $w_i$ is the lowest point of a simple loop of $dgm^{g,o}(h)$, then the surfaces $L_a$ and $L_b$ above $w_i$ are in the same connected component of $M_{i+1}$. We choose $g_3$ preserving or reversing the orientation of $M_i$ according to the orientability of the corresponding connected component of $f^{-1}(]f(v_i)-\epsilon,+\infty[)$.
    \item Else, $L_a$ and $L_b$ are not in the same connected components of $M_{i+1}$ and no orientation compatibility has to be fulfilled by $g_3$.
\end{itemize}

If $v_i$ is of type $G^+$, we do the same as for the case $J^+$ except that $\partial_{a_1}$ and $\partial_{a_2}$ are now in the same level surface component $L_a$. If $L_a$ is orientable, then we use the orientability condition or the non-orientability condition depending on whether $o=0$ or $o=1$. If $L_a$ is not orientable, then $o=1$ and whichever condition leads to the same result (as gluing a torus or a Klein bottle on a non-orientable surface has the same effect).

If $v_i$ is of type $S^-$, let $\partial_{a_1}$ and $\partial_{a_2}$ be above $v_i$ and $L_a$ be above $w_i$. Choose a rectangle $R$ in $L_a$ which joins $\partial_{a_1}$ and $\partial_{a_2}$ and call $\beta_1$ and $\beta_2$ the sides of $R$ that are not in $\partial_{a_1}$ or in $\partial_{a_2}$.\\
We set ${M_i = M_{i+1} \cup_{g_1} (\overline{L_{i+1}\backslash R}) \times I \cup_{g_4} MOD(S^-)}$ where $g_4$ maps $(\beta_1 \cup \beta_2) \times I$ to the sides $B_1$ and $B_2$ of $MOD(S^-)$.

\begin{wrapfigure}[11]{r}[34pt]{210px}
\captionsetup{justification=centering}
\centering
\caption{\label{ProofSample2} Possible Rectangles in Case $G^-$\\On the left, it leads to a non-orientable surface\\On the right, it leads to an orientable surface}
\includegraphics[width=200px]{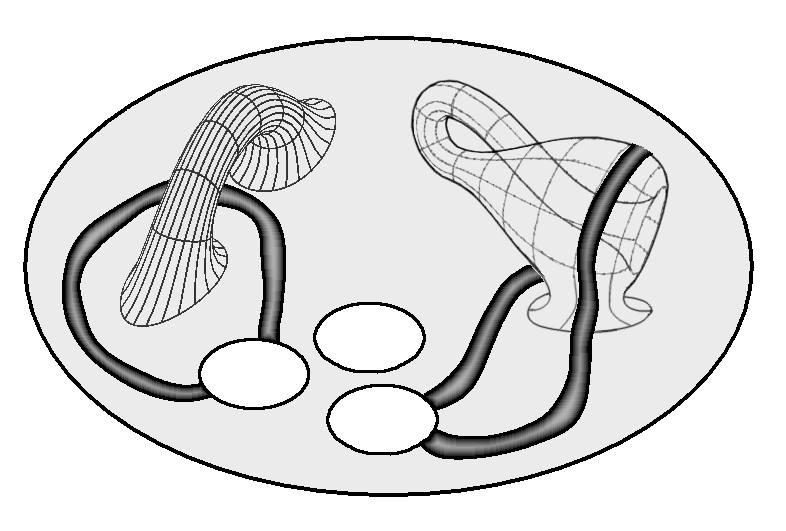}
\end{wrapfigure}

If $v_i$ is of type $G^-$, we do the same as for the case $S^-$ but with $R$ joining the curve above $v_i$ to itself. The rectangle $R$ must also verify two additional conditions:
\begin{itemize}
    \item $L_a \backslash R$ is connected,
    \item \begin{itemize}
         \item If $o'=0$, then we choose $R$ such that any Möbius strip (and thus orientation-reversing loops) of $L_a$ intersects each of the sides $\beta_1$ and $\beta_2$ of $R$,
         \item Else, we choose $R$ such that there exists a Möbius strip in $L_a$ disjoint from $R$.
         \end{itemize}
\end{itemize}

If $v_i$ is of type $O^-$, we do the same as for the case $G^-$ but with the 2$^{nd}$ additional condition becoming:
\begin{itemize}[label=--]
\item If $L_a$ is non-orientable and $o=0$, then $g+1$ is odd and we choose $R$ cutting all the Möbius strips of $L_a$,
\item Else, $g \ge 2$ and we choose $R$ such that $L_a \backslash R$ is non-orientable.
\end{itemize}

If $v_i$ is of type $J^-$, we do the same as for the case $G^-$ but we replace both additional conditions by these:
\begin{itemize}
\item $L_a \backslash R$ has two connected componants $L_x$ and $L_y$,
\item The demigenus and the orientability of $L_x$ match with the labels of one of the two edges below $w_i$ and so do $L_y$'s ones with the labels of the other edge.
\end{itemize}

If $v_i$ is of type $N^+$, let $\partial_a$ be above $v_i$ and $L_a$ be above $w_i$. Choose an annulus $A$ in $L_a$ such that $\partial_a$ is one of the boundary components of $A$.\\
We set ${M_i = M_{i+1} \cup_{g_1} (\overline{L_{i+1}\backslash A}) \times I \cup_{g_5} MOD(N^+)}$ where $g_5$ maps ${\overline{(L_{i+1}\backslash A)}\backslash (L_{i+1} \backslash A) \times I}$ to the cylinder side of $MOD(N^+)$.

If $v_i$ is of type $N^-$, let $L_a$ be above $w_i$.
Since $L_a$ is a disc, we can set \linebreak[4]{}
${M_i = M_{i+1} \cup_{g_1} (L_{i+1} \times I) \cup_{g_6} MOD(N^-)}$ where $g_6$ maps $L_a$ to the top of $MOD(N^-)$.

Proceeding all the way up to the minimal critical point, the constructed $M_1$ and $F_1$ are the non-singular extension of $f$ that we sought.
\end{proofthm}

\vspace{0.5cm}
Now the manifolds are $Kl^2$ bounding $\Kl$ and less Reeb graphs are possible than in the general situation, for both $dgm^+(f)$ and $dgm^{g,o}(F)$. Indeed, the number of Möbius points and the possible fundamental groups are limited for $dgm^+(f)$. Furthermore, the disposition of the non-orientable level surfaces and the possible fundamental groups are limited for the Reeb graph of any non-singular extension (or even for any non-singular Morse function) $F$ over $\Kl$. It turns out that all the situations for which a non-singular extension is possible are one of the followings:

\begin{table}[H]
\centering
\caption{\label{KleinPossiblePattern} Possible non-singular extension patterns}
\begin{tabular}{|r!{\vrule width0.005pt}c!{\vrule width0.005pt}l|}
\hline
$\pi_1(dgm^+(f)) \simeq \mathbb{Z}$ and &
\multirow{4}{*}{\begin{tikzpicture}
\draw[->,>=stealth] (0, 1) -- (1.5, 1);
\draw[->,>=stealth] (0, 0) -- (1.5, 0);
\draw[->,>=stealth] (0, 0.85) -- (1.5, 0.15);
\end{tikzpicture}}
& $\pi_1(dgm^{g,o}(F)) \simeq \mathbb{Z}$ and \\
$dgm^+(f)$ has no Möbius points &  & all the level surfaces are orientable\\
\Xcline{1-1}{0.01pt}\Xcline{3-3}{0.01pt}
$dgm^+(f)$ is simply connected &  & $dgm^{g,o}(F)$ is simply connected and the edges \\
and has two Möbius points &  & labelled with $o=1$ form a non-empty connected set\\
\hline
\end{tabular}
\end{table}

\noindent An example of the situation « $\pi_1(dgm^+(f)) \simeq \mathbb{Z} \rightarrow \pi_1(dgm^{g,o}(F)) \simeq \mathbb{Z}$ » is the figure \ref{CollapseSample}.\\
An example of the situation « $\pi_1(dgm^+(f)) \simeq \mathbb{Z} \rightarrow \pi_1(dgm^{g,o}(F)) \simeq \lbrace 0\rbrace$ » is the figure \ref{Example3}.\\
An example of the situation « $\pi_1(dgm^+(f)) \simeq \lbrace 0\rbrace \rightarrow \pi_1(dgm^{g,o}(F)) \simeq \lbrace 0\rbrace$ » is the figure \ref{Example1}.

We will not fully prove the left part of the table \ref{KleinPossiblePattern}, that is to say those are the only two possible situations for the Reeb graph of a germ of a Morse function along $Kl^2$. It is not needed for proving the Theorem \ref{ThmKleinVersion}. However, this result is interesting enough by itself so we give a partial proof of it.

For that purpose, we introduce the following definition.

\begin{definition}
Let $\mathcal{C}(Kl^2)$ and $\mathcal{C}(\Kl)$ be the sets of simple loops of $Kl^2$ and $\Kl$ respectively, then let:
\begin{itemize}
\item $\ell: (dgm(f) \backslash \{\text{vertices}\}) \rightarrow \mathcal{C}(Kl^2)$ mapping the points of $dgm(f)$ to the corresponding level curve connected components of $f$,
\item $E$ be an embedding of $Kl^2$ in $\Kl$ such that $E(Kl^2) = \partial\Kl$,
\item $\Pi: \mathcal{C}(\Kl) \rightarrow \pi_1(\Kl)$ mapping the simple loops onto one of the corresponding two classes in $\pi_1(\Kl)$.
\end{itemize}

We say that $f$ \emph{induces no loop} if $\Pi \circ E \circ \ell=0$. In short, we see the level curves of $f$ as loops in the solid Klein bottle. This definition does not depend on the chosen embedding because, given a presentation $\left\langle a,b \mid a b = b^{-1}a \right\rangle$ of $\pi_1(Kl^2)$, the generator $b$ is the one sent to $0 \in \pi_1(\Kl)$ for any choice of $E$. Any automorphism of $Kl^2$ keeps $\Ker(E_\#)$ invariant.

We say that $f$ \emph{induces one loop} if the set $\mathscr{G} = \overline{(\Pi \circ E \circ \ell)^{-1}(\mathbb{Z}^*)}$, where the closure is taken in $dgm(f)$, of level curves that are not contractible in $\Kl$ is a connected and simply connected subset of $dgm(f)$.
\end{definition}

Note that the properties of $f$ inducing one loop or inducing no loop cannot a priori be read on the Reeb graph of $f$ only. The following Proposition is a result very specific to the Klein bottle.

\begin{proposition}
$f$ either induces one loop or induces no loop.\\
Moreover, \guillemotleft $f$ induces one loop $\Leftrightarrow f$ admits at least one Möbius point $\Leftrightarrow f$ admits exactly two Möbius points \guillemotright.
\end{proposition}

\begin{wrapfigure}[10]{r}{160px}
\captionsetup{justification=centering}
\centering
\caption{\label{MobiusCurve} Level Curves around a Möbius Point}
\includegraphics[width=120px]{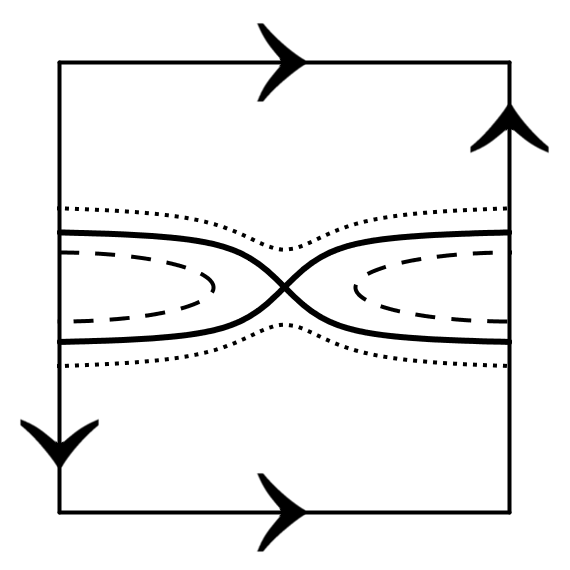}
\end{wrapfigure}

\hspace{0pt}
\begin{proof}
Let us first note that, since the boundary is a Klein bottle, there are at most two Möbius points.

Indeed, the level curve component of $f$ at a Möbius point is a connected curve $\mathscr{C}$ in $Kl^2$ satisfying:
\begin{itemize}
\item $Kl^2 \backslash \mathscr{C} = A \coprod B$ with $A$ and $B$ connected, $A$ orientable and $B$ non-orientable,
\item $\mathscr{C}$ is a regular curve except at the critical point which is a double point.
\end{itemize}
Since there can be at most two mutually disjoint curves of this type in a Klein bottle, there are at most two Möbius points for $f$.

\guillemotleft $f$ induces one loop $\Rightarrow f$ admits at least one Möbius point \guillemotright :
suppose that $f$ induces one loop. Then we construct an injective path on $dgm(f)$ starting at $\gamma \in \mathscr{G}$ and going to whichever direction. At (non-Möbius) saddle points, we choose the direction staying in $\mathscr{G}$. We can do that because:
\begin{itemize}
\item level curves in $\mathscr{G}$ are not contractible in $Kl^2$ and so we never reach critical points of types $N$ or $M$,
\item simple loops on the boundary of $\Kl$ are of one of the homotopy types $0$, $\pm 1$ and $\pm 2$ and so $\Image(\Pi \circ E \circ \ell) \subset \left\lbrace 0, \pm 1, \pm 2 \right\rbrace$,
\item if a level curve is of type $0 \in \pi_1(\Kl)$ above (resp. below) a Möbius point, then the level curves below (resp. above) that point are of types $2$ or $-2$,
\item at non-Möbius saddles points, the homotopy class of a level curve in the branche that is single (above a critical point of case $S^+$ or below a critical point of case $S^-$ for example) is either the sum or the difference of the classes of level curves in the two other branches.
\end{itemize}
And so $\Image(\Pi \circ E \circ \ell) \subset \left\lbrace 0, \pm 2 \right\rbrace$ and, at non-Möbius saddles points, either exactly $0$ or exactly $2$ branches correspond to points in $\mathscr{G}$. The constructed path eventually leads to a Möbius point since $dgm(f)$ is compact. It does not loop to its starting point $\gamma$ - otherwise there would be several level curves intersecting each other - hence it is simply connected.

\guillemotleft $f$ admits at least one Möbius point $\Rightarrow f$ admits two Möbius points $\Rightarrow f$ induces a loop\guillemotright :
suppose that $f$ admits a Möbius point $p$, then we also construct a path, this time starting at $p$ and going to the direction included in $\mathscr{G}$. By the same argument, this path leads to another Möbius point. This path contains all the points of $\mathscr{G}$ since it is a connected component of $\mathscr{G}$ and another connected component would imply the existence of two other Möbius points.

The equivalence \guillemotleft $f$ induces no loop $\Leftrightarrow f$ has no Möbius point\guillemotright\, is trivial from what precedes. It implies that $f$ either induces one loop or induces no loop.
\end{proof}

\begin{wrapfigure}[6]{r}{140px}
\vspace{-0.6cm}
\captionsetup{justification=centering}
\centering
\caption{\label{DiffPath} Allowable Collapse\\The dotted line is the set $\mathscr{G}$}
\includegraphics[width=120px]{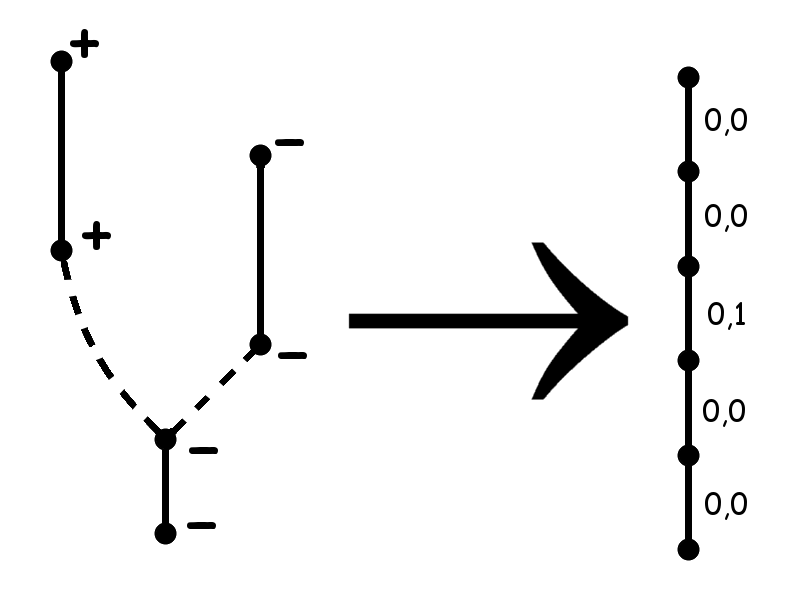}
\end{wrapfigure}

\textbf {Remark:}
When there is a non-singular extension $F$, the path $\mathscr{G}$ of this proposition does not always consist of the level curves of $f$ bounding non-orientable level surfaces of $F$. There are counter-examples of two kinds for that: the figure \ref{DiffPath} is a counter-example, and the crossway situation from table \ref{KleinPossiblePattern} gives other counter-examples.

\vspace{1.75cm}
\begin{proofthm}[\ref{ThmKleinVersion}]

$\implies$ : Let $F$ be a non-singular extension of $f$ to $\Kl$. We consider $dgm^{g,o}(F)$, the Reeb graph of $F$ labelled with both the demigenus and the orientability flag of its level surface components. We have the following commutative diagram:

\begin{center}
\begin{tikzcd}
Kl^2 \arrow[r, "E"] \arrow[d, "R_f"]
& \Kl \arrow[d, "R_F"] \\
dgm^+(f) \arrow[r, "C"]
& dgm^{g,o}(F)
\end{tikzcd}
\end{center}

$dgm^{g,o}(F)$ being a retract of $\Kl$, we have $\pi_1(dgm^{g,o}(F)) \in \lbrace \lbrace 0\rbrace, \mathbb{Z} \rbrace$.

Suppose that $\pi_1(dgm^{g,o}(F)) \simeq \mathbb{Z}$. Then $R_{F\,\#}$ is an isomorphism from $\pi_1(\Kl)$ onto $\pi_1(dgm^{g,o}(F))$. In particular, no non-trivial loop of $\pi_1(\Kl)$ is contained in a level surface component of $F$. Thus the label $o$ is $0$ for all the edges of $dgm^{g,o}(F)$.

Now suppose that $\pi_1(dgm^{g,o}(F)) \simeq \lbrace 0\rbrace$. We will distinguish three cases:
\begin{enumerate}
\item If there exists a surface level $S$ of odd demigenus $g$, then this level surface is non-orientable.
\item If the level surfaces all have an even demigenus and some of them have a non zero demigenus, then there are as many critical points of type $G^+$ as points of type $G^-$. Switching the value of the label $o$ between couples of them if necessary, we can find an allowable collapse satisfying the sought condition ($o=1$ for a non-empty connected component of the diagram).
\item The level surfaces may not all have a zero demigenus. Indeed, in such case, we can shrink the level surfaces: there is a space $R$ and a deformation retraction $\Kl \rightarrow R$ that sends each level surface component $S$ of $F$ onto the cell complex $e_0 \cup_{i=1}^{b-1} S^1$ where $b$ is the number of boundary components of $S$ and $e_0$ is a point (actually, since $\pi_2(\Kl) \simeq \lbrace 0\rbrace$, we can shrink $R$ further to $dgm^{g,o}(F)$). Every loop from such a cell complex can be moved up or down to a point and so we have ${\mathbb{Z} \simeq \pi_1(\Kl) \simeq \pi_1(R) \simeq \pi_1(dgm^{g,o}(F)) \simeq \lbrace 0\rbrace}$.
\end{enumerate}

In both of the first two cases, we found an collapse verifying $o=1$ for a non-empty connected component of $dgm^{g,o}(F)$.

The fact that the collapse must be allowable follows from the Theorem \ref{ThmAbstractVersion}.

\vspace{0.75cm}
$\impliedby$ : The construction of the extension is the same as for Theorem \ref{ThmAbstractVersion}. However, as for the orientable case\cite{Curley}, we must be more selective in choosing the rectangles at the $G^-$ stages, specifically those for which the label $o$ is preserved (we say that a critical point of type $G$ \emph{preserves} the label $o$ if $o = o'$ in the local labelled collapse of that critical point).

There can be at most two points of type $G$ switching the label $o$. For instance, there is exactly one point of type $G$ switching the label $o$ if and only if there are two Möbius points of the same type. In any case, because of the balance between the increases and the decreases of the label $g$, we are left with as many points of type $G^+$ preserving the label $o$ as points of type $G^-$ preserving the label $o$.

In the following, the notations are the ones used in the proof of Theorem \ref{ThmAbstractVersion}.

At the first stage corresponding to points $v_i$ of type $G^+$ preserving the label $o$, we choose two simple closed curves $\sigma_i$ and $\sigma^*_i$ in the level surface $L_i$. We choose the pair so that they intersect transversely in one point and so that $\sigma_i$ can be displaced upward to a boundary component of a level surface. We then displace those curves downward when constructing the following stages $M_j$. At other stages of type $G^+$ preserving the label $o$, we choose two other curves $\sigma_j$ and $\sigma^*_j$ disjoint from the displaced curves and satisfying the same conditions. At stages corresponding to a point $v_j$ of type $G^-$ preserving the label $o$, the rectange $R$ used for the construction of $M_j$ must also cut one $\sigma_i$ along a small arc, and none of the other chosen curves.

Proceeding this way, we avoid creating unnessecary non-trivial loops and so $\pi_1(M_1) \simeq \mathbb{Z}$. The following Lemma concludes the proof of the Theorem.

\begin{lemma}
Let $V^3$ be a compact connected 3-manifold with $\partial V^3=Kl^2$ and $\pi_1(V^3) \simeq \mathbb{Z}$. Then $V^3$ is the solid Klein bottle.
\end{lemma}

\begin{proof}
Let us consider $i: \pi_1(Kl^2) \rightarrow \pi_1(V^3)$. Since $\pi_1(V^3)$ is abelian and $\pi_1(Kl^2)$ is not, we have $\Ker(i) \ne 0$. By the loop Theorem\cite{Hempel}, there is a disk $D^2 \subset V^3$ such that $\partial D^2 \subset Kl^2$ and the simple loop $[\partial D^2]$ is not trivial in $Kl^2$.

We now prove that $[D^2] \ne 0$ in the relative homology group $H_2 (V^3, Kl^2 ; \mathbb{Z}/\raisebox{-0.3ex}{$2\mathbb{Z}$})$. Indeed, if $[\partial D^2]=0$ in $H_1(Kl^2 ; \mathbb{Z}/\raisebox{-0.3ex}{$2\mathbb{Z}$})$, then $Kl^2 \backslash \partial D^2$ has two connected components and none of them is a disc. Hence $Kl^2 \backslash \partial D^2$ are two Möbius strips. Then $[D^2] \ne 0$ in $H_2 (V^3, Kl^2 ; \mathbb{Z}/\raisebox{-0.3ex}{$2\mathbb{Z}$})$ since, otherwise, we would have \textit{Mö}$\cup D^2 \simeq \mathbb{P}^2(\mathbb{R})$ be the boundary of a connected component of $V^3 \backslash D^2$ and we already saw that it is not possible.

By the Poincaré-Lefschetz duality Theorem\cite{Lefschetz}, the intersection form $H_2 (V^3, Kl^2 ; \mathbb{Z}/\raisebox{-0.3ex}{$2\mathbb{Z}$}) \times H_1 (V^3 ; \mathbb{Z}/\raisebox{-0.3ex}{$2\mathbb{Z}$}) \rightarrow \mathbb{Z}/\raisebox{-0.3ex}{$2\mathbb{Z}$}$ is non-degenerate. In other words, if $S^1$ represents a generator of $\pi_1 (V^3) \simeq H_1 (V^3 ; \mathbb{Z})$, we have that $\# (D^2 \cap S^1)$ is odd, thus non-zero.

Consequently, $V^3 \backslash D^2$ is simply connected and $\partial (\overline{V^3 \backslash D^2})$ is the 2-sphere.
By the Poincaré-Perelman Theorem\cite{Perelman}, $\overline{V^3 \backslash D^2}$ is a 3-ball so that $V^3$ is realized as the non-orientable 3-ball with one handle, that is $\Kl$.

\end{proof}
\end{proofthm}

\paragraphtitle{Future Work}

The results of this article are a generalization of the problem of finding non-singular Morse functions extending a germ along a non-orientable 2-dimensional boundary. It also steps in the problem of controlling the domain of the extension when we target a non-contractible manifold.

We think that the method of the Theorem \ref{ThmKleinVersion} can be adapted at least to the torus, with an extraneous subtility. Indeed, the problems \guillemotleft Given a germ along the torus, find an non-singular extension to a solid torus\guillemotright\space and \guillemotleft Given a germ along $T^2$ and an embedding of $T^2$ as the boundary of $T_s^3$, find a non-singular extension compatible with the embedding\guillemotright\space are not equivalent anymore. This is because there are essentially two different embeddings of $T^2$ as the boundary of $T_s^3$, depending on whether we want to see $T_s^3$ as $S^1 \times D^2$ or $D^2 \times S^1$. In the case of the Klein bottle, there is no such ambiguity and we have a way to distinguish the loops that become trivial in the domain of the extension. The Reeb graph will likely not be enough in order to answer the second problem.

Another generalization is obviously to increment the dimension. Our approach suggests that the topology of the level hypersurfaces should be tracked in order to know which kind of transformations we are allowed to perform on them at each step. The diversity of the unlabelled local collapse configurations fastly becomes quite poor relatively to the dimension (in particular if we only consider generic Morse functions as we did here). For instance, the unlabelled local collapses one dimension higher is given in the table \ref{TabCollapseDim4}.

\begin{table}
\centering
\captionsetup{justification=centering}
\caption{\label{TabCollapseDim4} Possible Local Collapses of the Reeb graph of a germ along a $4-$Manifold\\They are all possible even if we restrict ourselves to orientable manifolds}
\includegraphics[width=400px]{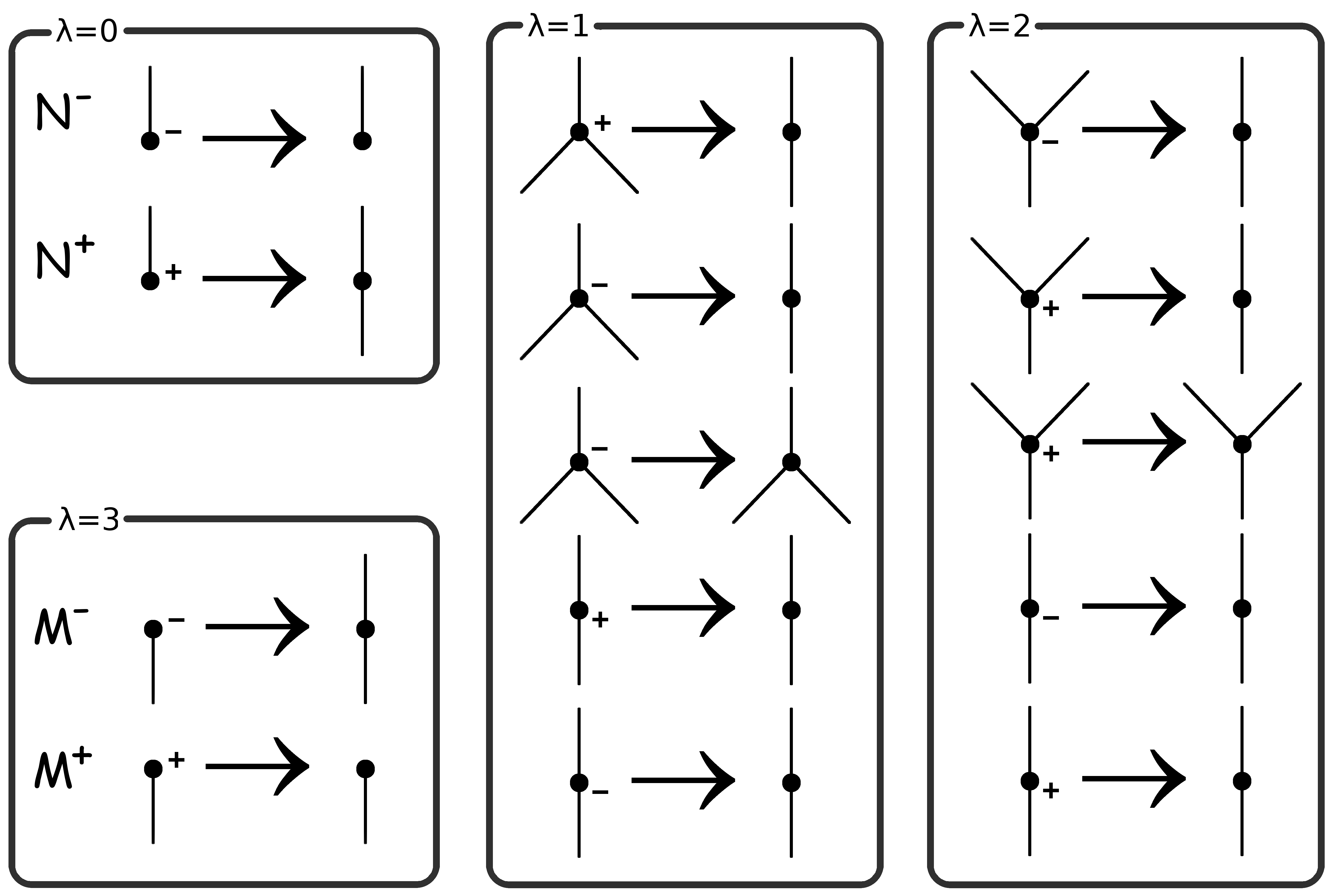}
\end{table}

M. Yamamoto\cite{Yamamoto} also gave results on this topic, one dimension lower: given a germ of a non-generic Morse function along $\partial S^1$, he found conditions for the existence of a non-singular extension $F$ to $D^2$ such that $F$ can be realized as an orthogonal projection in $R^2$. The situation is a bit different in that dimension, but we can expect, seeing his approach, that handling non-generic Morse function is not a totally trivial step anyway. Indeed, in that dimension, the approach used by Blank and Laudenbach\cite{BlankLaudenbach} consists of linking false extrema to true extrema for allowing the gradient lines coming from false extrema to arrive at critical points on the boundary, and not inside the domain of extension. With non-generic Morse functions, this linking process is more troublesome to make as we cannot link a false extremum to a true extremum belonging to the same level curve.

Lastly, the problem of finding non-singular extensions of germ of Morse functions can be used to describe the possible Morse functions on some closed manifold. Indeed, it is easy to see that searching for a non-singular Morse functions on a manifold with boundary is equivalent to searching for a Morse functions whose critical points are all on the boundary. Consider a closed manifold $M^n$ that can be realized as $M^n = N^n \cup_f N^n$ with $N^n$ being a manifold with boundary and $f$ gluing them along their boundary. A description of the non-singular Morse functions on $N^n$ gives a description of the Morse functions on $M^n$ (whose critical points are located in a sub-manifold $B^{n-1} \simeq \partial N^n$). The fact that the Theorem \ref{ThmAbstractVersion} and the other similar results in this area do not require the manifold to be connected is appreciable here ($\partial N^n$ does not have to be connected).

As an application of this last remark, we can see that the figures \ref{CollapseSample} and \ref{Example1} are the simplest extendable germs of Morse functions along $Kl^2$ in the sense that they both give rise to the Morse functions over the twisted product $S^2 \times S^1 \simeq \Kl \cup \Kl$ with the minimal number of critical points (four of them). We have an interesting relation between the topology of the level surfaces and the ordering of the Morse indices:\\
In the figure \ref{CollapseSample}, the critical points of the extension over the twisted $S^2 \times S^1$ are, from bottom to top, of Morse indices 0, 2, 1 and 3 and the level surfaces between them are diffeomorphic to $S^2$, $S^2 \coprod S^2$ and $S^2$.\\
In the figure \ref{Example1}, however, the Morse indices from bottom to top are 0, 1, 2 and 3 and the level surfaces in between are diffeomorphic to $S^2$, $Kl^2$ and $S^2$.

\renewcommand{\abstractname}{Acknowledgements}
\begin{abstract}
This work was made during the master degree of the author, financed by the university Paul-Sabatier of Toulouse, France.

\end{abstract}

\nocite{*}
\bibliographystyle{plain}
\bibliography{MorseExtension}

\end{document}